\documentclass[12pt]{article}

\usepackage{graphicx}
\usepackage{comment}
\usepackage{mathrsfs}

\usepackage{nameref}
\usepackage{blkarray}
\usepackage[shortlabels,inline]{enumitem}

\usepackage{amsmath}
\usepackage{amssymb}
\usepackage{empheq}
\usepackage{tikz}

\usepackage[shortlabels]{enumitem}
\setlist[enumerate]{nosep}

\usepackage[doc]{optional}
\usepackage{xcolor}
\usepackage[colorlinks=true,
            linkcolor=refkey,
            urlcolor=lblue,
            citecolor=lblue]{hyperref}
\usepackage{float}
\usepackage{soul}
\usepackage{graphicx}
\definecolor{labelkey}{rgb}{0,0.08,0.45}
\definecolor{refkey}{rgb}{0,0.6,0.0}
\definecolor{Brown}{rgb}{0.45,0.0,0.05}
\definecolor{lime}{rgb}{0.00,0.8,0.0}
\definecolor{lblue}{rgb}{0.5,0.5,0.99}

\usepackage{mathpazo}
\usepackage{blkarray}
\usepackage{subcaption}
\colorlet{hlcyan}{cyan!30}

\usepackage{stmaryrd}
\usepackage{amssymb}
\makeatletter
\def\namedlabel#1#2{\begingroup
   \def\@currentlabel{#2}%
   \label{#1}\endgroup
}
\makeatother

  \renewcommand{\epsilon}{\varepsilon}
  \renewcommand{\phi}{\varphi}

  \def\llangleT{\langle\kern-3.5pt\langle}
  \def\rrangleT{\rangle\kern-3.5pt\rangle}

  \newcommand{\defeq}{\coloneq}

  \makeatletter
  \def \weaktostar@sym{\setbox0=\hbox{$\rightharpoonup$}\rlap{\hbox 
  to\wd0{\hss\raise1ex\hbox{$\scriptscriptstyle{*\,}$}\hss}}\box0}
  \def \weaktostar    {\mathrel{\weaktostar@sym}}
  \makeatother

  \DeclareMathOperator{\ri}{ri}

  \def\DerivConeSym{V}
  \def\DerivConeDSym{W}
\newcommand{\DerivCone}[1][\relax]{\DerivConeSym\ifx\relax#1\relax\else(#1)\fi}
\newcommand{\DerivConePrime}[1][\relax]{\tilde\DerivConeSym\ifx\relax#1\relax\else(#1)\fi}
\newcommand{\DerivConeX}[2][]{\DerivConeSym_{#2}\ifx\relax#1\relax\else(#1)\fi}

\newcommand{\DerivConeDXt}[3][]{\DerivConeDSym_{#3}^{#2}\ifx\relax#1\relax\else(#1)\fi}
\newcommand{\DerivConeDXtp}[3][]{\DerivConeDSym_{#3}^{\circ,#2}\ifx\relax#1\relax\else(#1)\fi}

\newcommand{\localRhoX}[1][]{\if\relax\detokenize{#1}\relax r_x\else r_{x,#1}\fi}
\newcommand{\localRhoY}[1][]{\if\relax\detokenize{#1}\relax r_y\else r_{y,#1}\fi}

\def\weaklim{\mathop{\operatorname{w-\kern.07em lim}\,}}
\def\weaklimsup{\mathop{\operatorname{w-\kern.07em lim\,sup}\,}}
\def\weakliminf{\mathop{\operatorname{w-\kern.07em lim\,inf}\,}}
\def\weakstarlim{\mathop{\operatorname{w-\!\ast\!-\kern.07em lim}\,}}
\def\weakstarlimsup{\mathop{\operatorname{w-\!\ast\!-\kern.07em lim\,sup}\,}}
\def\weakstarliminf{\mathop{\operatorname{w-\!\ast\!-\kern.07em lim\,inf}\,}}

\newcommand{\menge}[2]{\big\{{#1}~\big |~{#2}\big\}}

\newcommand{\Menge}[2]{\left\{{#1}~\Big|~{#2}\right\}}

\newcommand{\fenv}[1]%
{\ensuremath{\,\overrightarrow{\operatorname{env}}_{#1}}}
\newcommand{\benv}[1]%
{\ensuremath{\,\overleftarrow{\operatorname{env}}_{#1}}}

\DeclarePairedDelimiterX\set[2]{ \{ }{ \}_{#2} }{#1}

\DeclarePairedDelimiterX\rb[1]{ ( }{ ) }{#1}

{\begin{list}{}{%
\settowidth{\labelwidth}{\textrm{#1~}}%
\setlength{\leftmargin}{\labelwidth+\labelsep}}}
{\end{list}}
\usepackage{amsthm}
\usepackage[capitalize,nameinlink]{cleveref}
\crefname{equation}{}{equations}
\crefname{chapter}{Appendix}{chapters}
\crefname{item}{}{items}
\crefname{enumi}{}{}

\theoremstyle{definition}
\newtheorem{theorem}{Theorem}[section]
\newtheorem{lemma}[theorem]{Lemma}

\newtheorem{corollary}[theorem]{Corollary}

\newtheorem{proposition}[theorem]{Proposition}

\newtheorem{ex}[theorem]{Example}

\newtheorem{remark}[theorem]{Remark}

\usepackage{multirow}

\usepackage[customcolors]{hf-tikz}
\usetikzlibrary{calc, intersections}
\usepackage{sidecap}
 \usepgflibrary{shapes.multipart}
 \usetikzlibrary{decorations.pathreplacing}
\usetikzlibrary{fadings}
\usepgflibrary{decorations.text}

\tikzset{style green/.style={
    set fill color=green!50!lime!60,
    set border color=white,
  },
  style cyan/.style={
    set fill color=cyan!90!blue!60,
    set border color=white,
  },
  style gray/.style={
    set fill
    color=black!10!,
    set border color=white,
  },
  style orange/.style={
    set fill color=orange!80!red!60,
    set border color=white,
  },
  hor/.style={
    above left offset={-0.15,0.31},
    below right offset={0.15,-0.125},
    #1
  },
  ver/.style={
    above left offset={-0.1,0.3},
    below right offset={0.15,-0.15},
    #1
  }
}

\newcommand{\boxedeqn}[1]{%
    \begin{align}\fbox{%
        \addtolength{\linewidth}{-2\fboxsep}%
        \addtolength{\linewidth}{-2\fboxrule}%
        \begin{minipage}{\linewidth}%
        \begin{equation}#1\\begin{align}+5mm]\end{equation}%
        \end{minipage}%
      }\end{align}%
  }

\providecommand{\ri}{\operatorname{ri}}

\usepackage{algorithm,algorithmic} 
\usepackage{hyperref}

\definecolor{myblue}{rgb}{.8, .8, 1}
  \newcommand*\mybluebox[1]{%
    \colorbox{myblue}{\hspace{1em}#1\hspace{1em}}}

\newcommand*\circled[2][1.6]{\tikz[baseline=(char.base)]{
    \node[shape=circle, draw, inner sep=1pt, 
        minimum height={\f@size*#1},] (char) {\vphantom{WAH1g}#2};}}
\makeatother

\allowdisplaybreaks 

\newcommand{\Iproj}{\overleftarrow{\pi}}     
\newcommand{\Rproj}{\overrightarrow{\pi}}    

\usepackage[customcolors]{hf-tikz}
\tikzset{style green/.style={
    set fill color=green!50!lime!60,
    set border color=white,
  },
  style cyan/.style={
    set fill color=cyan!90!blue!60,
    set border color=white,
  },
  style gray/.style={
    set fill
    color=black!10!,
    set border color=white,
  },
  style orange/.style={
    set fill color=orange!80!red!60,
    set border color=white,
  },
  hor/.style={
    above left offset={-0.15,0.31},
    below right offset={0.15,-0.125},
    #1
  },
  ver/.style={
    above left offset={-0.1,0.3},
    below right offset={0.15,-0.15},
    #1
  }
}

\definecolor{mygreen}{rgb}{0.5, 1.0, 0.83}
  \newcommand*\mygreenbox[1]{%
    \colorbox{mygreen}{\hspace{1em}#1\hspace{1em}}}

\definecolor{mygrey}{rgb}{0.7,0.75,0.71}

\begin{document}

\title{\textsc{
Backpropagation from KL Projections: Differential and Exact I-Projection Correspondences
}}

\author{
Manish\ Krishan Lal\thanks{
                 Munich Center for Machine Learning and Department of Mathematics, CIT,	
Technical University of Munich,\ Garching, 85748, Germany.
                 E-mail: \texttt{manish.krishanlal@tum.de}.}
                 }

\date{\today} 
\maketitle

\begin{abstract} \noindent
We establish two correspondences between reverse-mode automatic
differentiation (backpropagation at a given forward-pass point) and
compositions of projection maps in Kullback--Leibler (KL) geometry.
In both settings, message passing enforces agreement and factorization
constraints through KL projections.

In the first setting, backpropagation arises as the differential of a KL
projection map on a lifted deterministic computation graph.
In the second setting, on complete and decomposable sum--product networks,
the same reverse-mode quantities coincide with exact probabilistic marginals
and are realized by a KL I-projection.

The distinction is that, in the first setting, projection induces structure,
whereas, in the second, structure makes the projection exact. This pedagogical note highlights the relation among backpropagation,
belief propagation, and KL projection algorithms and provides a perspective that unifies learning, sampling, and inference under a common
geometric operator.
\end{abstract}



\section{Introduction}
\label{sec:introduction}

Modern deep learning is often analyzed through optimization and
dynamical-systems perspectives that study the evolution of network
parameters during training. These perspectives draw on tools from
stochastic approximation, curvature and Hessian analysis, and theories
of implicit regularization and overparameterization.

The \emph{backpropagation} algorithm is most naturally
understood as reverse-mode automatic differentiation applied to an
\emph{arithmetic (or algorithmic) circuit}. In this viewpoint, a neural
network is a directed acyclic graph whose nodes represent elementary
operations, and derivatives are propagated through this circuit by the
chain rule. This perspective underlies both computational-graph
implementations of automatic differentiation~\cite{baydin2018automatic}
and classical complexity results on circuit differentiation~\cite{parberry1994circuit}.

The question we consider here is different: can the
backpropagation computation itself be understood geometrically, as a
form of message passing on the computation graph, and can learning,
sampling, and inference also be studied under the same geometric
operator?

Our starting point is the observation that normalized
\emph{belief propagation} (BP) admits a projection interpretation in
Kullback--Leibler geometry. Walsh and Regalia~\cite{walsh2010bp} showed
that belief propagation can be viewed as a hybrid Bregman--Dykstra
projection method in which agreement and factorization constraints are
enforced through alternating KL projections.

A related observation was made in~\cite{eaton2022backprop}, where a
computation graph is embedded into a factor graph whose factors encode
the deterministic operations of the computation. Running loopy belief
propagation on this graph produces messages that encode the same
reverse-mode derivatives computed by backpropagation. In this sense,
backpropagation can be recovered as a special case of belief
propagation on an appropriately constructed graphical model. 

The contribution of the present paper is to show that
backpropagation admits two projection-theoretic interpretations within
KL geometry. One is differential: reverse-mode derivatives arise as the
differential of a KL projection map. The other is exact: in structured
probabilistic models, the same reverse-mode quantities coincide with
exact marginal probabilities and are realized by a KL
I-projection~\cite{csiszar1975idivergence}.

On the algorithmic side, our formulation also clarifies how to compare
gradient-based and constraint-based training schemes. In both
full-batch and mini-batch settings, a “training step” corresponds to a
single parameter update based on a given batch of data. In
gradient-based methods such as ADAM, this update is obtained by a
forward--backward pass on the batch. In constraint-based methods, the
update is obtained by enforcing local constraints and architectural
consensus across the same batch, possibly through multiple internal
projection iterations, as in divide-and-concur and related
projection-based learning frameworks.

Another main contribution is to make this comparison precise at the level
of information propagation: we show that a gradient-based training step
corresponds to evaluating the differential of a KL projection
operator, whereas projection-based methods iterate the operator
itself. Thus the two paradigms act on the same underlying geometric
object, but at different levels---evaluation versus linearization.
\subsection{Geometric setting}
\label{subsec:geometric-setting}

Throughout the paper, we work on finite probability simplices equipped
with KL Bregman geometry.
Messages, beliefs, and region tables are taken to be strictly positive
probability vectors, hence they lie in the relative interiors of the
relevant simplices.
We pass to log-coordinates whenever convenient.
The projections that appear are the Bregman I-projections and the
reverse-KL projections induced by the negative-entropy generator.
These objects are defined precisely in \cref{sec:prelim}.

There are two underlying scalar-valued computations in the paper.
In the first, one has a $C^1$ map
\[
F:\mathbb{R}^d\to\mathbb{R},
\qquad
z \coloneqq F(x),
\]
evaluated at an input $x$.
In the second, one has a network polynomial
\[
S:(0,\infty)^N\to(0,\infty),
\]
evaluated at evidence $e$.

To best describe both settings in a common language, we use the \emph{standard
region formalism} from variational message passing
\cite{yedidia2005constructing}.
Each node $n$ carries a scope $\mathrm{sc}(n)$, namely the set of
variables below that node, and the family of all such scopes forms a
multiset $\mathcal R$, which we call the region family.
Agreement between overlapping regions is expressed by consensus
constraints, meaning that equal-scope marginals must coincide.
Factorization across decomposed regions is expressed by product
constraints.
Depending on the setting, the state space is either a single lifted
joint object subject to these constraints, or a family of region-wise
objects coupled by them.

\subsection{Two regimes}
\label{subsec:regimes}

We study two regimes.

\begin{enumerate}[leftmargin=2em,itemsep=0.35em]

\item[\textup{(A)}]
\emph{Delta-lifted deterministic computation.}
Let $z \coloneqq F(x)$ be the scalar output of a smooth computation DAG. Each assignment $y=\psi(\cdot)$ in the DAG is replaced by a Dirac
constraint $\delta(y-\psi(\cdot))$, and one attaches a positive $C^1$ output factor $\phi:\mathbb{R}\to(0,\infty)$. This produces a lifted factor graph of the usual type used in
continuous message passing \cite{yedidia2005constructing}. On this lifted graph, normalized belief propagation is represented by
the KL projection operator
\begin{empheq}[box=\mybluebox]{equation*}
\mathcal T_{\mathrm{Proj}}
\;\defeq\;
\Rproj^{f}_{\widehat{\mathcal P}}
\circ
\Iproj^{f}_{\mathcal Q},
\end{empheq}
where $\mathcal Q$ is the consensus set for replicated variables and
$\widehat{\mathcal P}$ is the dual image of the relevant product
family; see \cref{sec:prelim,prop:WR}.
Thus $\mathcal T_{\mathrm{Proj}}$ acts on a strictly positive lifted
joint table and enforces first agreement, then factorization.

If $b_\downarrow(x)$ denotes the product of the downward BP messages
entering a variable $x$, then at a fixed forward evaluation one has
\begin{empheq}[box=\mygreenbox]{equation*}
\nabla_x \log b_\downarrow(x)
=
\frac{\phi'(z)}{\phi(z)}\,\nabla_x z.
\end{empheq}
Accordingly, the reverse-mode derivative is obtained by differentiating
the projection computation, more precisely by differentiating the
message readout associated with $\mathcal T_{\mathrm{Proj}}$.
This is the differential correspondence we prove in
\cref{thm:caseA}.

\item[\textup{(B)}]
\emph{Complete and decomposable sum--product networks.}
Here the model is a complete and decomposable SPN with strictly
positive leaf indicators
\cite{darwiche2003differential,poon2011spn}. Write $\lambda=(\lambda_{i,t})_{i,t}\in(0,\infty)^N$ for the indicator values associated with variables $X_i$ taking state
$t$, and identify the evidence $e$ with $\lambda$.
Thus $e\coloneqq\lambda$, and $N=\sum_{i=1}^M |\mathcal X_i|$. For fixed positive circuit weights, the SPN defines the multilinear map $S:(0,\infty)^N\to(0,\infty)$,
called the network polynomial. If $r$ is the root node, then $S(e)=S(r)$. Each node $n$ also carries an upward value $S(n)$ and a reverse-mode
quantity
\[
D(n)\coloneqq \frac{\partial S(r)}{\partial S(n)}.
\]

In this regime, the reverse-mode derivatives with respect to the leaf
indicators recover exact marginals.
In particular,
\begin{empheq}[box=\mygreenbox]{equation*}
\frac{\partial S(e)}{\partial \lambda_{i,t}}
=
\Pr(X_i=t\mid e)\,S(e).
\end{empheq}
Thus the backward sweep is not merely analogous to inference; it
computes exact probabilistic marginals after normalization by $S(e)$. The same computation is realized by the two-step KL projection operator
\begin{empheq}[box=\mybluebox]{equation*}
\mathcal T_{\mathrm{SPN}}
\;\defeq\;
\Rproj^{f_{\mathcal R}}_{\widehat{\mathcal P}_{\mathcal R}}
\circ
\Iproj^{f_{\mathcal R}}_{\mathcal Q_{\mathcal R}},
\end{empheq}
where $\mathcal Q_{\mathcal R}$ is the consensus set on region tables
and $\widehat{\mathcal P}_{\mathcal R}$ is the dual image of the
region-wise product family.
This is the exact I-projection correspondence of the paper; see
\cref{thm:spn-derivatives-are-marginals,prop:spn-two-step}.

\end{enumerate}

In Regime~(A), the region family $\mathcal R$ records the replicated
copies introduced by the delta lift.
In Regime~(B), it is simply the family of SPN scopes.

\subsection{Relation to prior work}
\label{subsec:prior-work}

The KL projection view of belief propagation was presented in Walsh and
Regalia \cite{walsh2010bp}.
For arithmetic circuits and SPNs, the relevant differential semantics
go back to Darwiche \cite{darwiche2003differential}.
The convex-analytic background comes from the classical theory of
Bregman projections
\cite{csiszar1975idivergence,bauschke2002forward}.
The delta-lift identity relating belief propagation and backpropagation
already appears in \cite{eaton2022backprop}.
These algorithms can be placed in a
single projection-theoretic framework, with a differential
correspondence in Regime~(A) and an exact I-projection correspondence
in Regime~(B). Moreover, loopy
belief propagation admits variational interpretations: its fixed points
correspond to stationary points of the Bethe free energy and can be
studied through CCCP-type formulations, where messages act as Lagrange
multipliers updated from current belief estimates
\cite{yuille2002cccp,yurtsever2022cccp}. 

The remainder of the paper is organized as follows.
\cref{sec:prelim} develops the notation and KL geometry.
\cref{sec:caseA} treats Regime~(A).
\cref{sec:caseF} treats Regime~(B). \cref{sec:discussion} concludes.

\section{Preliminaries and notation}
\label{sec:prelim}

For a scalar function $h$, we write $\nabla h$ for its gradient.
For a vector-valued map $F$, we write $JF$ for its Jacobian.
The vector $\mathbf 1$ denotes the all-ones vector of the relevant dimension.
We write $a\propto b$ when $a$ and $b$ agree up to a positive,
state-independent constant, followed by normalization to return to a simplex.

\paragraph{Probability simplices and relative interiors.}
Let $\mathcal{X}$ be a finite set with $|\mathcal{X}|=n$.
We identify probability distributions on $\mathcal{X}$ with points of the simplex
\[
\Delta(\mathcal{X})
\;\defeq\;
\menge{p\in\mathbb{R}^n_{\ge 0}}{\mathbf{1}^\top p = 1}.
\]
Its relative interior is $\Delta^\circ(\mathcal{X})
\defeq
\menge{p\in\Delta(\mathcal{X})}{p_i>0\ \forall i}$. Since $\Delta(\mathcal{X})$ lies in the affine hyperplane
$\{p\in\mathbb{R}^n \mid \mathbf{1}^\top p = 1\}$, it has empty interior
as a subset of $\mathbb{R}^n$. Accordingly, all smooth operations used in the paper,
including logarithms, gradients, and dual coordinates, are understood on the
relative interior $\ri(\Delta(\mathcal{X}))=\Delta^\circ(\mathcal{X})$.

\paragraph{Bregman divergences and KL geometry.}
Let $f:\mathcal D\to\mathbb R$ be strictly convex and differentiable
on a convex domain $\mathcal D\subseteq\mathbb R^n$.
For $r\in\mathcal D$ and $q\in\ri(\mathcal D)$, the associated Bregman divergence is
\[
D_f(r,q)
\;\defeq\;
f(r)-f(q)-\langle \nabla f(q),\,r-q\rangle.
\]
In general $D_f(r,q)\in[0,+\infty]$, and the value $+\infty$
may occur outside the effective domain; see
\cite{bauschke1997legendre,bauschke2017convex}.
When $f$ is of Legendre type, the Fenchel conjugate $f^*$ is
well defined, $\nabla f$ identifies primal and dual coordinates, and
\[
D_f(r,q)=D_{f^*}(\nabla f(q),\nabla f(r)).
\]
In the present paper, the relevant generator is the negative entropy
\[
f(p)\;\defeq\;\sum_i p_i\log p_i,
\qquad
p\in\Delta^\circ(\mathcal{X}),
\]
with conjugate $f^*(\theta)=\log\sum_i e^{\theta_i}$. Thus $\nabla f(p)=\bigl(1+\log p_i\bigr)_i$, and $\nabla f^*(\theta)=\mathrm{softmax}(\theta)\in\Delta^\circ(\mathcal{X})$, and $D_f(r,q)=\mathrm{KL}(r\|q)$.

\paragraph{Left (I-) and right (reverse-KL) projections.}
Let $C\subseteq\Delta(\mathcal{X})$ be nonempty, convex, and
closed in the relative topology of the simplex.
We write $\widehat C \;\defeq\; \nabla f(C)\subset\mathbb R^n$ for the image of $C$ in dual coordinates.

The left, or I-projection, is
\[
\Iproj^f_C(q)
\;\in\;
\arg\min_{r\in C} D_f(r,q)
\;=\;
\arg\min_{r\in C}\mathrm{KL}(r\|q).
\]

The right projection is
\[
\Rproj^f_{\widehat C}(q)
\;\defeq\;
\nabla f^*\!\Bigl(
  \arg\min_{u\in\widehat C} D_{f^*}\bigl(u,\nabla f(q)\bigr)
\Bigr).
\]
When $\widehat C=\nabla f(C)$, this is the usual reverse-KL
$M$-projection of $q$ onto $C$, written in dual coordinates
\cite{csiszar1975idivergence}. For KL geometry, strict convexity of $f$
ensures uniqueness on $\Delta^\circ(\mathcal{X})$.

\paragraph{Product families.}
Let $X_1,\dots,X_M$ be variables with finite alphabets
$\mathcal{X}_1,\dots,\mathcal{X}_M$, and set
\[
\Omega\;\defeq\;\prod_{i=1}^M \mathcal{X}_i.
\]
The corresponding joint simplex is $\Delta(\Omega)$.
The product family is
\[
\mathcal C
\;\defeq\;
\Menge{
  \prod_{i=1}^M q_i(x_i)
}{
  q_i\in\Delta^\circ(\mathcal{X}_i)
}
\;\subset\;
\Delta^\circ(\Omega),
\]
and its dual image is $\widehat{\mathcal C}\defeq \nabla f(\mathcal C)\subset\mathbb R^{|\Omega|}$. Right projections onto such sets enforce factorization:
they replace a joint object by the closest product object in reverse-KL
geometry.
In Regime~(A) and Regime~(B), the corresponding factorization sets are
denoted by $\mathcal P$ and $\mathcal P_{\mathcal R}$, respectively.

\paragraph{Factor graphs.}
A factor graph is a bipartite graph
$G=(\mathcal V,\mathcal F,\mathcal E)$.
Its variable nodes are the $X_i\in\mathcal V$, its factor nodes are the
$g\in\mathcal F$, and $(g,X_i)\in\mathcal E$ when the factor $g$
depends on the variable $X_i$.
Each factor is a strictly positive function
\[
g:\prod_{i\in\mathsf{nb}(g)} \mathcal{X}_i \to (0,\infty),
\]
where $\mathsf{nb}(u)$ denotes the set of neighbors of a node $u$.

\paragraph{Messages and normalized BP.}
For each edge $(g,X_i)\in\mathcal E$, belief propagation carries two messages
\[
m_{g\to i}\in\Delta^\circ(\mathcal{X}_i),
\qquad
m_{i\to g}\in\Delta^\circ(\mathcal{X}_i).
\]
The normalized sum--product updates are
\begin{align*}
m_{g\to i}(x_i)
&\propto
\sum_{x_{\mathsf{nb}(g)\setminus i}}
  g(x_{\mathsf{nb}(g)})\prod_{j\in\mathsf{nb}(g)\setminus i} m_{j\to g}(x_j),\\
m_{i\to g}(x_i)
&\propto
\prod_{h\in\mathsf{nb}(i)\setminus g} m_{h\to i}(x_i),
\end{align*}
followed by normalization so that the resulting messages again lie in
$\Delta^\circ(\mathcal{X}_i)$.

\paragraph{Beliefs and gauge invariance.}
The belief at a variable node $X_i$ is
\[
b_i(x_i)
\;\propto\;
\prod_{g\in\mathsf{nb}(i)} m_{g\to i}(x_i)
\;\in\;
\Delta^\circ(\mathcal{X}_i).
\]
Messages are only defined up to positive, state-independent scalings on each
edge.
All identities used later are written in terms of log-derivatives, and hence
are invariant under this gauge freedom.

\paragraph{Consensus sets.}
Suppose a variable $X$ is replicated into $K$ copies
$X^{(1)},\dots,X^{(K)}$, each taking values in the same alphabet
$\mathcal{X}$.
The associated consensus set is the diagonal face
\cite[Chapter~26]{bauschke2017convex}
\[
\mathcal D
\;\defeq\;
\Menge{
  q\in\Delta(\mathcal{X}^K)
}{
  q(x^{(1)},\dots,x^{(K)})=0
  \ \text{unless}\ 
  x^{(1)}=\cdots=x^{(K)}
}.
\]
Thus $\mathcal D$ consists exactly of those distributions supported on the
diagonal.
It is a closed convex face of $\Delta(\mathcal{X}^K)$.
The I-projection $\Iproj^f_{\mathcal D}$ enforces consensus by removing
off-diagonal mass and renormalizing on the diagonal.
In Regime~(A) and Regime~(B), the corresponding consensus sets are denoted by
$\mathcal Q$ and $\mathcal Q_{\mathcal R}$, respectively.

\section{Regime~(A): Delta-lifted backpropagation as the differential of a KL projection algorithm}
\label{sec:caseA}

We begin with the projection-theoretic description of normalized belief propagation due to Walsh and Regalia~\cite{walsh2010bp}. We then pass to the smooth delta-lift and show that the reverse-mode derivatives are read off from the downward messages.

\subsection{Projection algorithm: BP as a hybrid of KL projections}
\label{subsec:hybrid-proj}

Let $G=(\mathcal V,\mathcal F,\mathcal E)$ be a discrete factor graph with strictly
positive factors. Following the Walsh--Regalia replication construction,
let $\Omega$ be the finite alphabet of one replica of the variables, so that lifted joint tables
live on $\Omega^K$ for some replication index $K$.

Two constraint sets enter the construction. Let $\mathcal Q\subset \Delta(\Omega^K)$ be the diagonal face, and let
$\Delta^\circ(\mathcal Q)$ denote its relative interior, that is, the set of
probability tables supported on the diagonal and strictly positive there.
Let $\mathcal P$ be the family of product distributions on the replicated variables
in the sense of~\cite{walsh2010bp}, and write
\[
\widehat{\mathcal P}\defeq \nabla f(\mathcal P)\subset\mathbb R^{|\Omega|^K}
\]
for its image in dual coordinates.

\paragraph{Projections with faces.}
The I-projection $\Iproj^f_{\mathcal Q}$ enforces agreement among the replicas: it removes all mass off the diagonal and renormalizes on the diagonal. Since points of $\Delta^\circ(\mathcal Q)$ vanish away from the diagonal, dual coordinates on $\mathcal Q$ are understood relative to the affine hull of $\mathcal Q$. Equivalently, one may identify $\Delta(\mathcal Q)$ with the simplex on the diagonal outcomes and use the same entropy generator on that smaller simplex. With this convention, the right projection $\Rproj^f_{\widehat{\mathcal P}}$ is well defined on $\Delta^\circ(\mathcal Q)$.

For $q\in\Delta^\circ(\Omega^K)$ define the two-step operator
\begin{equation*}
\label{eq:TProj-def}
\mathcal{T}_{\mathrm{Proj}}(q)
\;\defeq\;
\Rproj^f_{\widehat{\mathcal P}}\!\big(\,\Iproj^f_{\mathcal Q}(q)\,\big).
\end{equation*}
Thus $\mathcal T_{\mathrm{Proj}}$ first imposes consensus and then imposes product structure in dual coordinates. The key point of this section is that backpropagation does not merely resemble message passing: it is obtained by differentiating the projection operator $\mathcal T_{\mathrm{Proj}}$. Thus the backward pass is the linearization of a KL projection computation.

\paragraph{Dykstra’s hybrid algorithm with Bregman Projections.}
Let $q_{-1}\in\Delta^\circ(\Omega^K)$ be the normalized product of the replicated
factor tables. Walsh and Regalia~\cite[Thm.~1]{walsh2010bp} show that normalized
BP is represented by the following Dykstra-type hybrid iteration:
\begin{align*}
k_n
&= \Rproj^f_{\widehat{\mathcal P}}\!\Big(\nabla f^{*}\big(\nabla f(q_{n-1})+\sigma_{n-1}\big)\Big),\\
\sigma_n
&= \nabla f(q_{n-1})+\sigma_{n-1}-\nabla f(k_n),\\
r_n
&= \Iproj^f_{\mathcal Q}\!\Big(\nabla f^{*}\big(\nabla f(k_n)+\tau_{n-1}\big)\Big),\\
q_n
&= \Rproj^f_{\widehat{\mathcal P}}(r_n),\\
\tau_n
&= \nabla f(k_n)+\tau_{n-1}-\nabla f(q_n),
\end{align*}
initialized with $\sigma_{-1}=\mathbf{0}$ and $\tau_{-1}=\mathbf{0}$.

\begin{proposition}[Normalized BP as a hybrid of KL projections {\cite[Thm.~1]{walsh2010bp}}]
\label{prop:WR}
For a finite factor graph $G$ with strictly positive factors and initialization $q_{-1}$ as above, the iterates of the above algorithm coincide with normalized sum--product (BP) on $G$.
In particular:
\begin{enumerate}[label=(\roman*),leftmargin=2em]
\item if $G$ is a tree, one upward--downward BP pass computes $\mathcal{T}_{\mathrm{Proj}}(q_{-1})$, and the resulting beliefs are exact marginals;
\item any fixed point of normalized BP is a fixed point of $\mathcal{T}_{\mathrm{Proj}}$.
\end{enumerate}
\end{proposition}

For the purposes of Regime~(A), \cref{prop:WR} is used only as a representation theorem. On a tree it identifies one BP sweep with the projection map $\mathcal T_{\mathrm{Proj}}$. The backpropagation identity proved below is local and does not depend on any convergence statement for the WR iterates.

\subsection{Delta factors and a message-level chain rule}
\label{subsec:delta-chain}

We now pass from the discrete projection picture to the smooth delta-lift. Let $y=\psi(x_1,\ldots,x_r)$ with
$\psi\in C^1$ on an open set
$\mathcal U\subset\mathbb{R}^{d_1}\times\cdots\times\mathbb{R}^{d_r}$,
and consider the deterministic factor
\[
g(y,x_1,\ldots,x_r) \;=\; \delta\big(y-\psi(x_1,\ldots,x_r)\big).
\]

Assume the incoming message $m_{y\to g}(y)$ is strictly positive and
$C^1$ near the forward value $y^*$, and that all other arguments
$x_{-j}$ are clamped to $x_{-j}^*$.

\begin{lemma}[Delta-factor chain rule]
\label{lem:delta-chain}
Under these assumptions,
\[
\left.\frac{\partial}{\partial x_j}\log m_{g\to x_j}(x_j)\right|_{x_j^*}
\;=\;
\left.\frac{\partial}{\partial y}\log m_{y\to g}(y)\right|_{y=y^*}\cdot
\left.\frac{\partial y}{\partial x_j}\right|_{x^*}.
\]
\end{lemma}

\begin{proof}
The delta constraint collapses the factor-to-variable update to a substitution:
\[
m_{g\to x_j}(x_j)
\;\propto\;
\int \delta\big(y-\psi(x_j,x_{-j}^*)\big)\, m_{y\to g}(y)\,\mathrm dy
=
m_{y\to g}\!\big(\psi(x_j,x_{-j}^*)\big).
\]
Taking logarithms and differentiating with respect to $x_j$ gives the claim.
\end{proof}

\subsection{Backpropagation as the differential of a KL projection map}
\label{subsec:main-identity}

Let $z$ denote the scalar output node of the computation DAG, and let
$\phi:\mathbb{R}\to(0,\infty)$ be a $C^1$ output factor.
Run one upward pass, so that the messages through delta factors collapse to Dirac
masses at the forward values, and then run one downward pass seeded by
\[
m_{\phi\to z}(z) \;=\; \phi(z).
\]
For each variable $v$ define the downward belief
\[
b_\downarrow(v)
\;\defeq\;
\prod_{g\;\in\;\mathrm{nb}(v)\;\cap \;\text{downstream}} m_{g\to v}(v),
\qquad
s(v)
\;\defeq\;
\left.\partial_v \log b_\downarrow(v)\right|_{v^*}.
\]

\begin{theorem}[Backpropagation as log-derivatives of downward BP messages]
\label{thm:caseA}
With the setup above, $s(z)=\phi'(z^*)/\phi(z^*)$, and for any variable $x$,
\[
\left.\nabla_x \log b_\downarrow(x)\right|_{x^*}
\;=\;
\frac{\phi'(z^*)}{\phi(z^*)}\ \left.\nabla_x z\right|_{x^*}.
\]
Equivalently, for $\phi(z)=e^{\alpha z}$ one has
$\nabla_x \log b_\downarrow(x)=\alpha\,\nabla_x z$ at the forward point $x^*$.
\end{theorem}

\begin{proof}
Let
\[
c\;\defeq\;\frac{\phi'(z^*)}{\phi(z^*)}.
\]
Then $s(z)=c$ by the choice of the seed message $m_{\phi\to z}=\phi$. At each deterministic factor $g$ corresponding to $y=\psi(\cdot)$,
\cref{lem:delta-chain} yields
\[
\partial_x\log m_{g\to x}
\;=\;
\partial_y\log m_{y\to g}\cdot \partial_x y
\quad\text{at the forward point}.
\]
Since $b_\downarrow(x)$ is the product of all factor-to-variable messages into $x$,
\[
s(x)
\;=\;
\partial_x\log b_\downarrow(x)
\;=\;
\sum_{y\in \mathrm{ch}(x)} s(y)\,\partial_x y,
\]
where $\mathrm{ch}(x)$ denotes the children of $x$ in the computation DAG. Thus the family $x\mapsto s(x)$ satisfies the same backward recursion as the family $x\mapsto c\,\nabla_x z$. Indeed, by the ordinary chain rule,
\[
c\,\nabla_x z
\;=\;
\sum_{y\in\mathrm{ch}(x)} c\,\nabla_y z\,\partial_x y.
\]
To conclude, choose a topological ordering $v_1,\dots,v_N$ of the DAG such that every edge points from a lower index to a higher index, and let $v_N=z$. We prove by backward induction on this order that
\[
s(v_k)=c\,\nabla_{v_k}z\qquad\text{for }k=N,N-1,\dots,1.
\]
For $k=N$ this is exactly the identity $s(z)=c=c\,\nabla_z z$. Suppose now that the identity holds for every child $y\in\mathrm{ch}(x)$ of a node $x=v_k$. Then
\[
s(x)
=
\sum_{y\in\mathrm{ch}(x)} s(y)\,\partial_x y
=
\sum_{y\in\mathrm{ch}(x)} c\,\nabla_y z\,\partial_x y
=
c\,\nabla_x z,
\]
where the last step is the chain rule above. This closes the induction and proves the theorem.
\end{proof}

\begin{corollary}[Layered backpropagation as the differential of local KL projection steps]
\label{cor:layered-geometric}
Consider a feedforward network
\[
a^{[0]}=x,\qquad
u^{[l]}=W^{[l]}a^{[l-1]}+b^{[l]},\qquad
a^{[l]}=\sigma_l(u^{[l]}),
\qquad l=1,\dots,L,
\]
with scalar output $z \;\defeq\; C(a^{[L]},y)$, and a positive output factor $\phi:\mathbb{R}\to(0,\infty)$. In the corresponding delta-lifted factor graph, each variable
$a^{[l]}$ and $u^{[l]}$ is replicated across the primitive factors in which
it participates. The consensus constraints identify these copies, while the
factorization constraints decouple the surrounding local message tables.
Define the message sensitivities
\[
s_a^{[l]}
\;\defeq\;
\left.\nabla_{a^{[l]}} \log b_\downarrow(a^{[l]})\right|_{\text{forward point}},
\qquad
s_u^{[l]}
\;\defeq\;
\left.\nabla_{u^{[l]}} \log b_\downarrow(u^{[l]})\right|_{\text{forward point}}.
\]
Then these quantities are well defined independently of the chosen copy and
propagate layerwise according to
\begin{align*}
\text{(output)}\qquad
s_a^{[L]}
&=
\left.\nabla_{a^{[L]}} \log \phi\big(C(a^{[L]},y)\big)\right|_{\text{forward point}},\\[0.5em]
\text{(activation layers)}\qquad
s_u^{[l]}
&=
\sigma_l'(u^{[l]}) \odot s_a^{[l]},\\[0.5em]
\text{(linear layers)}\qquad
s_a^{[l-1]}
&=
(W^{[l]})^\top s_u^{[l]},
\end{align*}
for $l=1,\dots,L$. Equivalently, if
\[
\mathcal G_{\mathrm{lin}}^{[l]}
\;\defeq\;
\bigl\{(a^{[l-1]},u^{[l]}) : u^{[l]}=W^{[l]}a^{[l-1]}+b^{[l]}\bigr\},
\]
and
\[
\mathcal G_{\sigma}^{[l]}
\;\defeq\;
\bigl\{(u^{[l]},a^{[l]}) : a^{[l]}=\sigma_l(u^{[l]})\bigr\},
\]
then $s_u^{[l]}=\sigma_l'(u^{[l]})\odot s_a^{[l]}$ is the differential of the local KL projection step associated with
$\mathcal G_{\sigma}^{[l]}$, while $s_a^{[l-1]}=(W^{[l]})^\top s_u^{[l]}$ is the differential of the local KL projection step associated with
$\mathcal G_{\mathrm{lin}}^{[l]}$. In particular, for $\phi(z)=e^{-z}$ and $\delta^{[l]}\defeq -\,s_u^{[l]}$,
one recovers the standard backpropagation recursion
\[
\delta^{[L]}
=
\sigma_L'(u^{[L]}) \odot \nabla_{a^{[L]}} C(a^{[L]},y),
\qquad
\delta^{[l]}
=
\sigma_l'(u^{[l]}) \odot (W^{[l+1]})^\top \delta^{[l+1]},
\]
together with the parameter gradients
\[
\nabla_{W^{[l]}} C
=
\delta^{[l]}(a^{[l-1]})^\top,
\qquad
\nabla_{b^{[l]}} C
=
\delta^{[l]}.
\]

Thus, in standard deep learning, the usual backward pass through layers is
precisely the propagation of consensus-synchronized log-derivatives of
downward BP messages. The familiar affine and nonlinear layer rules are the
factorization-decoupled differentials of local KL projection steps, and the
ordinary chain-rule recursion is the layered form of \cref{thm:caseA}.
\end{corollary}

\section{Regime~(B): Sum--product circuits --- backpropagation equals exact marginals and realizes a KL I-projection}
\label{sec:caseF}

In this regime the circuit has a genuine probabilistic interpretation: the forward pass evaluates the network polynomial, while the backward pass recovers exact marginal quantities. Thus, unlike Regime~(A), where the KL picture appears after differentiation, here the projection operator itself already produces the relevant probabilities. This shows that marginalization and differentiation arise from the same underlying KL mechanism, depending on whether one evaluates the operator or its differential.

Let $X_1,\ldots,X_M$ be discrete variables with finite alphabets $\mathcal{X}_i$, and set $\Omega \defeq \prod_{i=1}^M \mathcal{X}_i$. As in \cref{sec:prelim}, distributions on $\Omega$ are points of the simplex $\Delta(\Omega)$, and we work on its relative interior $\Delta^\circ(\Omega)$. Throughout this section all indicators and all weights are assumed strictly positive. In particular, every normalization and every logarithm that appears below is well defined.

\subsection{Sum--product networks and upward/downward values}

An SPN is a directed acyclic graph $G=(\mathcal N,\mathcal E)$ whose internal nodes are partitioned into \emph{sum} nodes $\mathsf{Sum}$ and \emph{product} nodes $\mathsf{Prod}$, and whose leaves are indicator functions of the form $\mathbf{1}[X_i=t]$ for $t\in\mathcal X_i$. We write $\mathrm{ch}(n)$ for the set of children of a node $n$.
Each node $n\in\mathcal N$ carries a \emph{scope} $\mathrm{sc}(n)\subseteq\{1,\ldots,M\}$, namely the set of variables appearing below $n$. We impose the usual structural assumptions:
\begin{itemize}[leftmargin=1.5em]
\item \textbf{Completeness:} if $s\in\mathsf{Sum}$ and $c\in\mathrm{ch}(s)$, then
\[
\mathrm{sc}(c)=\mathrm{sc}(s).
\]
\item \textbf{Decomposability:} if $p\in\mathsf{Prod}$ and $c\neq c'$ are children of $p$, then
\[
\mathrm{sc}(c)\cap \mathrm{sc}(c')=\varnothing.
\]
\item \textbf{Positive weights:} each edge $(s\to c)$ out of a sum node carries a weight
\[
w_{sc}\in(0,\infty).
\]
When $\sum_{c\in\mathrm{ch}(s)} w_{sc}=1$, we say that the sum node is \emph{normalized}.
\end{itemize}
Soft evidence is specified by strictly positive leaf values
$\lambda_{i,t}\in(0,\infty)$, for $t\in\mathcal X_i$. Let $\lambda = (\lambda_{i,t})_{i,t}\in(0,\infty)^N$, and $N\defeq\sum_{i=1}^M |\mathcal X_i|$, and identify the evidence vector $e$ with $\lambda$, writing
$e\coloneqq \lambda$. These numbers play the role of unary evidence potentials. In this soft-evidence setting, the posterior marginals are recovered from logarithmic derivatives. The \emph{network polynomial} is the map $S:(0,\infty)^N\to(0,\infty)$ obtained by the bottom-up rules
\begin{align*}
\text{Leaf } \ell=\mathbf{1}[X_i=t]:&\quad S(\ell)=\lambda_{i,t},\\[0.25em]
\text{Product } p:&\quad S(p)=\prod_{c\in\mathrm{ch}(p)} S(c),\\[0.25em]
\text{Sum } s:&\quad S(s)=\sum_{c\in\mathrm{ch}(s)} w_{sc}\,S(c).
\end{align*}
If $r$ denotes the root, we write $S(e)\;\defeq\;S(r)\;>\;0$.

\paragraph{Downward values.}
Define the \emph{downward value} at each node $n$ by
\[
D(n)\;\defeq\;\frac{\partial S(r)}{\partial S(n)},
\qquad
D(r)\equiv 1.
\]
Thus $S(n)$ is the forward value at node $n$, while $D(n)$ is the corresponding reverse-mode quantity. Since all weights and indicators are positive, every $S(n)$ is positive, and hence every derivative above is well defined.
At each sum node $s$ we shall define a local gating distribution $b_s\in\Delta^\circ(\mathrm{ch}(s))$, and at each variable $X_i$ a variable belief $b_i\in\Delta^\circ(\mathcal X_i)$. The explicit formulas appear in \cref{thm:spn-derivatives-are-marginals}.

\subsection{Derivatives equal marginals}
\label{subsec:spn-deriv-marg}

We now recall the differential semantics of arithmetic circuits in the present SPN setting. A \emph{parse} is obtained by choosing exactly one child at each sum node and all children at each product node. The positive weights and the evidence induce a probability distribution on parses.

\begin{theorem}[Log-derivatives equal marginals]
\label{thm:spn-derivatives-are-marginals}
Assume the SPN is complete and decomposable with positive weights and positive indicators.

\begin{enumerate}[label=(\alph*),leftmargin=2em,itemsep=0.25em]
\item \emph{Variable marginals.}
For each $i$ and each state $t\in\mathcal X_i$,
\[
\Pr(X_i=t\mid e)
\;=\;
\frac{\lambda_{i,t}}{S(e)}\,\frac{\partial S(e)}{\partial \lambda_{i,t}}
\;=\;
\frac{\partial \log S(e)}{\partial \log \lambda_{i,t}}
\in(0,1).
\]
Accordingly, define
\[
b_i(t)\;\defeq\;\frac{\lambda_{i,t}}{S(e)}\,\frac{\partial S(e)}{\partial \lambda_{i,t}}.
\]
Then $b_i\in\Delta^\circ(\mathcal X_i)$ and $b_i(t)=\Pr(X_i=t\mid e)$.

\item \emph{Gate marginals at sum nodes.}
For each sum node $s$ and each child $c\in\mathrm{ch}(s)$,
\[
\Pr(\text{parse visits }s \text{ and selects }c\mid e)
\;=\;
\frac{D(s)\,w_{sc}\,S(c)}{S(e)}.
\]
Define the local gating distribution
\[
b_s(c)\;\defeq\;\frac{w_{sc}\,S(c)}{S(s)}.
\]
Then $b_s\in\Delta^\circ(\mathrm{ch}(s))$, and $b_s(c)$ is the conditional probability of choosing $c$ given that the parse reaches $s$. More precisely,
\[
\Pr(\text{visit }s\mid e)\;=\;\frac{D(s)S(s)}{S(e)},
\qquad
\Pr(\text{select }c\mid e,\text{visit }s)\;=\;b_s(c).
\]
\end{enumerate}

Thus one upward pass, which computes the values $S(n)$, and one downward pass, which computes the values $D(n)$, recover all variable marginals and all sum-node gate marginals.
\end{theorem}

\subsection{Projection interpretation: a two-step KL I-projection}
\label{subsec:spn-projection}
We now reformulate the same computation in the projection language used in Regime~(A).

\paragraph{Region family and ambient space.}
Let $\mathcal R \;\defeq\; \menge{\mathrm{sc}(n)}{n\in\mathcal N}$ be the multiset of node scopes, and for each $R\in\mathcal R$ write
\[
\mathcal X_R \;\defeq\; \prod_{i\in R}\mathcal X_i .
\]
We work on the product space of region tables
\[
\mathcal D_{\mathcal R}
\;\defeq\;
\prod_{R\in\mathcal R}\Delta^\circ(\mathcal X_R),
\]
equipped with the separable negative-entropy generator
\[
f_{\mathcal R}(q)
\;\defeq\;
\sum_{R\in\mathcal R}\sum_{x_R\in\mathcal X_R} q_R(x_R)\log q_R(x_R),
\qquad
q=(q_R)_{R\in\mathcal R}\in\mathcal D_{\mathcal R},
\]
whose Bregman divergence is $D_{f_{\mathcal R}}(r,q)
= \sum_{R\in\mathcal R}\mathrm{KL}(r_R\|q_R)$.

\paragraph{Product and consensus constraints.}
Define the product family $\mathcal P_{\mathcal R}\subset\mathcal D_{\mathcal R}$ by
\[
\mathcal P_{\mathcal R}
\;\defeq\;
\Menge{
q\in\mathcal D_{\mathcal R}}{
q_{\mathrm{sc}(p)}
=
\bigotimes_{c\in\mathrm{ch}(p)} q_{\mathrm{sc}(c)}
\ \text{for all } p\in\mathsf{Prod}
}.
\]
Here $\bigotimes$ denotes the product distribution on the disjoint child scopes. In log-coordinates,
\[
\theta_R=\nabla f(q_R)=1+\log q_R,
\]
these constraints become affine, up to normalization, and therefore $\widehat{\mathcal P}_{\mathcal R}
\defeq
\nabla f_{\mathcal R}(\mathcal P_{\mathcal R})$ is an affine subset of dual coordinates. Define the consensus set $\mathcal Q_{\mathcal R}\subset\mathcal D_{\mathcal R}$ by $q_R=q_{R'}$
whenever $R,R'\in\mathcal R$ represent the same scope.

\paragraph{Two-step KL operator.}
Let $z_{-1}\in\mathcal D_{\mathcal R}$ be the collection of region tables induced by the SPN parameters and the evidence. Define
\begin{equation*}
\label{eq:spn-two-step}
\mathcal T_{\mathrm{SPN}}(z_{-1})
\;\defeq\;
\Rproj^{f_{\mathcal R}}_{\widehat{\mathcal P}_{\mathcal R}}
\Bigl(
  \Iproj^{f_{\mathcal R}}_{\mathcal Q_{\mathcal R}}(z_{-1})
\Bigr).
\end{equation*}
This is the SPN analogue of the operator $\mathcal T_{\mathrm{Proj}}$ from Regime~(A).

\paragraph{Unfolding notation.}
If the SPN has shared subcircuits, let $\widetilde G=(\widetilde{\mathcal N},\widetilde{\mathcal E})$ be the rooted tree obtained by unfolding the DAG, and let $\rho:\widetilde{\mathcal N}\to\mathcal N$ be the natural surjection sending each copy to its original node. Each copy $\tilde n$ has the same node type, scope, and local parameters as $\rho(\tilde n)$; in particular,
\[
\mathrm{sc}(\tilde n)=\mathrm{sc}(\rho(\tilde n)).
\]
If $\tilde s$ is a copy of a sum node $s$ and $c\in\mathrm{ch}(s)$, we write $\tilde c$ for the child of $\tilde s$ corresponding to $c$. The edge $\tilde s\to\tilde c$ carries the same weight $w_{sc}$. Write $\widetilde{\mathcal R}$ for the multiset of scopes in $\widetilde G$.

\begin{lemma}[Tree case]
\label{lem:spn-tree-case}
Suppose the SPN has no shared subcircuits, so that its underlying node graph is a rooted tree. Then one upward--downward pass computes
\[
\mathcal T_{\mathrm{SPN}}
=
\Rproj^{f_{\mathcal R}}_{\widehat{\mathcal P}_{\mathcal R}}
\circ
\Iproj^{f_{\mathcal R}}_{\mathcal Q_{\mathcal R}},
\]
and the resulting region marginals coincide with the beliefs of
\cref{thm:spn-derivatives-are-marginals}.
\end{lemma}

\begin{proof}
In the tree case, the constraints encoded by $\mathcal Q_{\mathcal R}$ and $\mathcal P_{\mathcal R}$ are exactly the Walsh--Regalia consensus and product constraints on the associated factor graph. Hence $\mathcal T_{\mathrm{SPN}}
=
\Rproj^f_{\widehat{\mathcal P}}\circ\Iproj^f_{\mathcal Q}$. By \cref{prop:WR}, one upward--downward BP pass computes this operator and returns exact marginals; by \cref{thm:spn-derivatives-are-marginals}, those marginals are precisely the beliefs $b_i$ and $b_s$.
\end{proof}

\begin{lemma}[Unfolding preserves upward values]
\label{lem:spn-unfolding-upward}
For every $\tilde n\in\widetilde{\mathcal N}$,
\[
S_{\widetilde G}(\tilde n)=S_G(\rho(\tilde n)).
\]
In particular,
\[
S_{\widetilde G}(e)=S_G(e).
\]
\end{lemma}

\begin{proof}
We argue by induction from the leaves upward. The claim is immediate at leaf indicators. If $\tilde n$ is a product node, then its children are copies of the children of $\rho(\tilde n)$, so the induction hypothesis and the product recursion give
\[
S_{\widetilde G}(\tilde n)
=
\prod_{\tilde c\in\mathrm{ch}(\tilde n)} S_{\widetilde G}(\tilde c)
=
\prod_{c\in\mathrm{ch}(\rho(\tilde n))} S_G(c)
=
S_G(\rho(\tilde n)).
\]
The same argument with the sum recursion proves the result for sum nodes.
\end{proof}

\begin{lemma}[Downward values aggregate over copies]
\label{lem:spn-unfolding-downward}
Let $D_G$ and $D_{\widetilde G}$ denote the downward values on $G$ and $\widetilde G$, respectively. Then for every node $n\in\mathcal N$,
\[
D_G(n)=\sum_{\tilde n:\,\rho(\tilde n)=n} D_{\widetilde G}(\tilde n).
\]
\end{lemma}

\begin{proof}
We argue by reverse induction from the root downward. The claim is trivial at the root. Suppose it holds for all parents of $n$. Then
\[
D_G(n)
=
\sum_{u\in\mathrm{pa}_G(n)}
D_G(u)\,
\frac{\partial S_G(u)}{\partial S_G(n)}.
\]
Substituting the induction hypothesis gives
\[
D_G(n)
=
\sum_{u\in\mathrm{pa}_G(n)}
\sum_{\tilde u:\,\rho(\tilde u)=u}
D_{\widetilde G}(\tilde u)\,
\frac{\partial S_G(u)}{\partial S_G(n)}.
\]
Each copy $\tilde u$ has a unique child $\tilde n$ with $\rho(\tilde n)=n$, and \cref{lem:spn-unfolding-upward} implies
\[
\frac{\partial S_{\widetilde G}(\tilde u)}{\partial S_{\widetilde G}(\tilde n)}
=
\frac{\partial S_G(u)}{\partial S_G(n)}.
\]
Hence the sum above is exactly
\[
\sum_{\tilde n:\,\rho(\tilde n)=n} D_{\widetilde G}(\tilde n).
\]
\end{proof}

\begin{lemma}[Beliefs on the DAG are merged beliefs on the unfolding]
\label{lem:spn-merge-beliefs}
Under the identification of copied leaf indicators with the original evidence variables, the beliefs on the DAG are obtained by merging the corresponding beliefs on the unfolding. More precisely, for each variable $X_i$ and state $t\in\mathcal X_i$,
\[
\frac{\lambda_{i,t}}{S_G(e)}\frac{\partial S_G(e)}{\partial \lambda_{i,t}}
=
\sum_{\tilde\ell:\,\rho(\tilde\ell)=(i,t)}
\frac{\lambda_{\tilde\ell}}{S_{\widetilde G}(e)}
\frac{\partial S_{\widetilde G}(e)}{\partial \lambda_{\tilde\ell}},
\]
and for each sum node $s$ and child $c\in\mathrm{ch}(s)$,
\[
\frac{D_G(s)\,w_{sc}\,S_G(c)}{S_G(e)}
=
\sum_{\tilde s:\,\rho(\tilde s)=s}
\frac{D_{\widetilde G}(\tilde s)\,w_{sc}\,S_{\widetilde G}(\tilde c)}
     {S_{\widetilde G}(e)},
\]
where $\tilde c$ denotes the child of $\tilde s$ corresponding to $c$. Moreover,
\[
\frac{w_{sc}S_G(c)}{S_G(s)}
=
\frac{w_{sc}S_{\widetilde G}(\tilde c)}{S_{\widetilde G}(\tilde s)}
\qquad
\text{for every copy }\tilde s\text{ of }s.
\]
\end{lemma}

\begin{proof}
Under the identification
\[
\lambda_{\tilde\ell}=\lambda_{i,t}
\qquad
\text{whenever }\rho(\tilde\ell)=(i,t),
\]
the unfolded and original network polynomials agree:
\[
S_G(e)=S_{\widetilde G}(e).
\]
Differentiating with respect to the shared variable $\lambda_{i,t}$ yields
\[
\frac{\partial S_G(e)}{\partial \lambda_{i,t}}
=
\sum_{\tilde\ell:\,\rho(\tilde\ell)=(i,t)}
\frac{\partial S_{\widetilde G}(e)}{\partial \lambda_{\tilde\ell}}.
\]
Multiplying by $\lambda_{i,t}/S_G(e)$ gives the first identity. For the gate marginals, \cref{lem:spn-unfolding-downward} and \cref{lem:spn-unfolding-upward} imply
\[
D_G(s)
=
\sum_{\tilde s:\,\rho(\tilde s)=s} D_{\widetilde G}(\tilde s),
\qquad
S_G(c)=S_{\widetilde G}(\tilde c),
\qquad
S_G(e)=S_{\widetilde G}(e),
\]
which yields the second identity after multiplication by $w_{sc}S_G(c)/S_G(e)$. The final identity follows immediately from \cref{lem:spn-unfolding-upward}.
\end{proof}

\begin{theorem}[SPN sweep as a KL two-step projection]
\label{prop:spn-two-step}
For a complete and decomposable SPN with positive weights and evidence,
one upward--downward pass computes the operator
\[
\mathcal T_{\mathrm{SPN}}
=
\Rproj^{f_{\mathcal R}}_{\widehat{\mathcal P}_{\mathcal R}}
\circ
\Iproj^{f_{\mathcal R}}_{\mathcal Q_{\mathcal R}},
\]
and the resulting region marginals coincide with the beliefs of
\cref{thm:spn-derivatives-are-marginals}.
\end{theorem}

\begin{proof}
The tree case is exactly \cref{lem:spn-tree-case}. For a general DAG SPN, let $\widetilde G$ be the unfolding. By \cref{lem:spn-tree-case}, one upward--downward pass on $\widetilde G$ computes the analogous two-step KL projection on $\widetilde{\mathcal R}$, and its region marginals coincide with the beliefs on $\widetilde G$.

Passing from $\widetilde G$ back to $G$ amounts to identifying copies that represent the same original scope, which is precisely the consensus identification encoded by $\mathcal Q_{\mathcal R}$. The product constraints are local and unchanged by unfolding. Hence the region marginals of $\mathcal T_{\mathrm{SPN}}(z_{-1})$ are obtained by merging the corresponding region marginals from the unfolded tree. By \cref{lem:spn-merge-beliefs}, these merged marginals are exactly the variable and gate beliefs of \cref{thm:spn-derivatives-are-marginals}.
\end{proof}

\begin{remark}
The correspondence relies crucially on decomposability. When this condition fails, the network polynomial is no longer multilinear, and the projection--differentiation equivalence breaks down.
\end{remark}

\section{Conclusion}
\label{sec:discussion}

We have identified two KL-projection correspondences underlying backpropagation.

\paragraph{Outlook.}
We conclude with several directions that this perspective naturally motivates.

\begin{itemize}[leftmargin=1.5em,itemsep=0.25em]

\item \emph{\textbf{Learning} dynamics of gradient-based, hybrid, and projection-based models.}
The backward pass can be interpreted as structured signal propagation~\cite{schoenholz2017deep} induced by the linearization of a KL projection map. This connects classical analyses of signal propagation in gradient-based models~\cite{roberts2022principles} with projection-driven dynamics. In particular, one may study stability, implicit regularization, and directional convergence of the associated fixed-point systems, which are often nonconvex or inconsistent~\cite{alwadani2021behaviour,andrews2025augmented,dinh2025algorithmic,moursi2016douglas,singh2024strong}. 
From this viewpoint, learning dynamics are governed not only by the update rule but also by structural properties of the underlying feedforward computational graph, such as depth, connectivity, and path interactions, which control how signals propagate and interact. This perspective also suggests analyzing the existence, uniqueness, and stability of fixed points, as well as the global structure of solution landscapes and their transitions, beyond classical convergence guarantees. A natural conjectural direction is that learning exhibits a form of \emph{finite-window stabilization}: after sufficiently many training steps, forward computations, backward signals, and parameter updates become jointly stable over any fixed horizon, so that both representations and gradient signals change only negligibly over bounded windows.

\item \emph{\textbf{Sampling} and KL gradient flows.}
Many sampling procedures, including flow-based and diffusion models, can be viewed as dynamical systems whose stationary laws solve KL-type or nonconvex variational problems. From the projection viewpoint developed here, a natural conjecture is that such dynamics admit an implicit decomposition into alternating projection steps onto evolving constraint sets, with the associated differentials governing transport and stability. In this sense, sampling can be interpreted as evaluating or approximating the same projection operator that underlies inference and learning, but along a continuous-time or stochastic trajectory. This suggests a unified perspective in which convergence, mixing, and finite-time accuracy are controlled by the geometry of the underlying projection dynamics.

\item \emph{\textbf{Inference} via structured projection operators.}
A wide range of inference procedures can be interpreted as enforcing local consistency and factorization constraints through iterative updates. In the framework developed here, these updates arise naturally as alternating KL projections onto consensus and product families, mirroring the operators appearing in both regimes. This suggests that probabilistic inference, logical reasoning, and variational approximations can all be viewed as instances of evaluating a common projection operator on different state spaces. A natural conjectural direction is that lifting classical Euclidean projection-based methods to KL geometry yields hybrid continuous--discrete representations in which exact and approximate inference are unified through the same underlying projection mechanism.

\end{itemize}

Finally, this viewpoint suggests that learning, sampling, and inference are not separate paradigms but different manifestations of a common projection-based operator, allowing one to study any pair—or all three—within a unified geometric framework.

\paragraph{Acknowledgments.}
This work was presented in an invited talk, \emph{Learning, Sampling, and Inference Through the Lens of Projections}, at the Banff workshop \emph{Strategies for Handling Applications with Nonconvexity} (SHAWN), May 2025.
The author gratefully acknowledges partial travel support from the PIMS–BIRS–Simons Travel Award and the Alexander von Humboldt (AvH) Foundation, and thanks Suvrit Sra for helpful discussions.

\bibliographystyle{plain}
\bibliography{ref-bp}
\end{document}